\documentclass[11pt]{article}
\usepackage{latexsym}
\usepackage{amsmath}
\usepackage{amsthm,color}
\usepackage{amssymb}
\usepackage{mathrsfs}
\usepackage[pagewise]{lineno} 
\usepackage[colorlinks,  linkcolor=blue,  anchorcolor=blue, citecolor=blue]{hyperref}

\usepackage{lscape}
\newtheorem{theorem}{\indent Theorem}[section]

\newtheorem{lemma}[theorem]{\indent Lemma}
\newtheorem{remark}[theorem]{\indent Remark}
\newtheorem{case*}[theorem]{\indent Case*}
\begin{document}
\renewcommand{\baselinestretch}{1.3}


\begin{center}
    {\large \bf Critical system involving fractional Laplacian}
\vspace{0.5cm}\\
{\sc Maoding Zhen$^{1,2}$,\quad Jinchun He$^{1,2}$ and Haoyuan Xu*$^{1,2}$  }\\
{\small 1) School of Mathematics and Statistics, Huazhong University of Science and Technology,\\ Wuhan 430074, China}\\
{\small 2) Hubei Key Laboratory of Engineering Modeling and Scientific Computing, Huazhong University of Science and Technology,}
{\small Wuhan, 430074, China}
\end{center}


\renewcommand{\theequation}{\arabic{section}.\arabic{equation}}
\numberwithin{equation}{section}


\begin{abstract}
In this paper, we study the following critical system with fractional Laplacian:
\begin{equation*}
\begin{cases}

(-\Delta)^{s}u= \mu_{1}|u|^{2^{\ast}-2}u+\frac{\alpha\gamma}{2^{\ast}}|u|^{\alpha-2}u|v|^{\beta}  \ \ \  \text{in} \ \ \mathbb{R}^{n},\\

(-\Delta)^{s}v= \mu_{2}|v|^{2^{\ast}-2}v+\frac{\beta\gamma}{2^{\ast}}|u|^{\alpha}|v|^{\beta-2}v\ \ \ \ \text{in} \ \ \mathbb{R}^{n},\\

u,v\in D_{s}(\mathbb{R}^{n}).
\end{cases}
\end{equation*}

By using the Nehari\ manifold,\ under proper conditions, we establish the existence and nonexistence of positive least energy solution of the system.

\end{abstract}
\vspace{-1 cm}

\footnote[0]{ \hspace*{-7.4mm}
$^{*}$ Corresponding author.
The authors were supported by the NSFC grant 11571125.\\
E-mails: d201677010@hust.edu.cn; taoismnature@hust.edu.cn; hyxu@hust.edu.cn
}

\section{Introduction}
\hspace{0.6cm} Recently, a great attention has been focused on the study of equations or systems involving the fractional Laplacian\ with nonlinear terms, both for their interesting theoretical structure and their concrete
applications(see \cite{BCPS,CDS2,MF,CS1,CRS,QW,SV2,CS2,ZW,SZY,GLZ} and references therein). This type of operator arises in a quite natural way in many different contexts, such as, the thin obstacle problem, finance, phase
transitions, anomalous diffusion, flame propagation and many others(see \cite{GGP,DPV,MMC,LS} and references therein).

Compared to the Laplacian  problem, the fractional Laplacian problem is nonlocal and more challenging. In 2007, L. Caffarelli and L. Silvestre  \cite{CS1} studied an extension problem related to the fractional Laplacian in
${\mathbb R}^n$, which can transform the nonlocal problem into a local problem in ${\mathbb R}_+^{n+1}$. This method can be extended to bounded regions and is extensively used in recent articles. For example, B. Barrios, E.
Colorado, A. de Pablo and U. S\'{a}nchez \cite{BCPS} studied the following nonhomogeneous equation involving fractional Laplacian,
\begin{equation*}
\begin{cases}
(-\Delta)^{s}u= \lambda u^{q}+u^{\frac{n+s}{n-s}} $$ &\text{in} \ \ \Omega,\\

u=0& \text{on} \ \partial\Omega,\\
\end{cases}
\end{equation*}
and proved the existence and multiplicity of solutions under suitable conditions of $s$ and $q$. In the above, the fractional Laplacian operator $(-\Delta)^{s}$ is defined through the spectral decomposition using the powers of the eigenvalues of the positive Laplace operator $(-\Delta)$ with zero Dirichlet boundary data.

Furthermore, E. Colorado, A. de Pablo and U. S\'{a}nchez \cite{CDS2} studied the following fractional equation with critical Sobolev exponent
\begin{equation*}
\begin{cases}
(-\Delta)^{s}u=|u|^{2^{\ast}-2}u+f(x) $$  &\text{in}\  \ \Omega,\\

u=0& \text{on} \ \partial\Omega,\\
\end{cases}
\end{equation*}
where the existence and the multiplicity of solutions were proved under appropriate conditions on the size of $f$. For more recent advances on this topic, see \cite{MYL,YG,NT} and references therein.

It is also natural to study the coupled system of equations. X. He, M. Squassina and W. Zou \cite{XSZ} considered the following fractional Laplacian system with critical nonlinearities on a bounded domain in ${\mathbb R}^n
$
\begin{equation*}
\begin{cases}
(-\Delta)^{s} u= \lambda|u|^{q-2}u+{\frac{2\alpha}{\alpha+\beta}}|u|^{\alpha-2}u|v|^{\beta} $$  & \text{in} \ \ \Omega,\\
(-\Delta)^{s} v = \mu|v|^{q-2}v+{\frac{2\beta}{\alpha+\beta}}|u|^{\alpha}|v|^{\beta-2}v& \text{in} \ \ \Omega,\\

u=0,v=0& \text{on}\ \partial \Omega,
\end{cases}
\end{equation*}
using variational methods and a Nehari manifold decomposition, they prove that the system admits at least two positive solutions when the pair of parameters $(\lambda,\mu)$ belongs to certain subset of $\mathbb{R}^{2}$.

When\ $\Omega={\mathbb R}^n$, the Dirichlet-Neumann map used in  \cite{BCPS,CDS2,XSZ} provides a formula for the fractional Laplacian in the whole space, which is equivalent to that obtained from Fourier Transform
\cite{CS1}, where the operator has explicit expression,
$$
(-\Delta)^{s}u(x)=C(n,s)P.V.\int_{R^{n}}\frac{u(x)-u(y)}{|y-x|^{n+2s}}dy,
$$with

$$
C(n,s)=\left(\int_{R^{n}}\frac{1-\cos(\varsigma)}{|\varsigma|^{n+2s}}d\varsigma\right)^{-1}=2^{2s}\pi^{-\frac{n}{2}}\frac{\Gamma(\frac{n+2s}{2})}{\Gamma(2-s)}s(1-s),\quad 0<s<1.
$$

In \cite{GLZ}, Z. Guo, S. Luo and W. Zou studied the following critical system involving fractional Laplacian
\begin{equation}\label{int1}
\begin{cases}
(-\Delta)^{s}u= \mu_{1}|u|^{2^{\ast}-2}u+\frac{\alpha\gamma}{2^{\ast}}|u|^{\alpha-2}u|v|^{\beta} $$  &\text{in} \ \ \mathbb{R}^{n},\\

(-\Delta)^{s}v= \mu_{2}|v|^{2^{\ast}-2}v+\frac{\beta\gamma}{2^{\ast}}|u|^{\alpha}|v|^{\beta-2}v&\text{in} \ \ \mathbb{R}^{n},\\

u,v\in D_{s}(\mathbb{R}^{n}),
\end{cases}
\end{equation}
 they showed the existence of positive least energy solution, which is radially symmetric with respect to some point in $\mathbb{R}^{n}$\ and decays at infinity with certain rate. Q. Wang \cite{QW} studied a special case
 where $\alpha=\beta=\frac{2^{\ast}}{2}$ and\ $\frac{\alpha\gamma}{2^{\ast}}=\frac{\beta\gamma}{2^{\ast}}=\beta$,\ the author also showed the existence of positive least energy solution under suitable conditions.

 In this paper, we study the existence of the least energy solutions for the system \eqref{int1} with critical exponent. We assume

$$
 \mu_{1},\ \mu_{2}>0,\ 2^{\ast}=\frac{2n}{n-2s},\ n>2s,\ \alpha,\ \beta>1,\mbox{ and } \alpha+\beta=2^{\ast}.
$$

Let  ${D_{s}(\mathbb{R}^{n})}$\ be Hilbert space as the completion of $C^{\infty}_{c}(\mathbb{R}^{n})$
 equipped with the norm
$$
\|u\|^{2}_{D_{s}(\mathbb{R}^{n})}=\frac{C(n,s)}{2}\int_{\mathbb{R}^{n}}\int_{\mathbb{R}^{n}}\frac{|u(x)-u(y)|^{2}}{|y-x|^{n+2s}}dxdy.
$$

Let
\begin{equation}\label{int2}
S_{s}=\inf \limits_{u\in D_{s}(\mathbb{R}^{n})\setminus 0}\frac{\|u\|^{2}_{D_{s}(\mathbb{R}^{n})}}{(\int_{R^{n}}|u|^{2^{\ast}}dx)^{\frac{2}{2^{\ast}}}}
\end{equation}
be the sharp imbedding constant of ${D_{s}(\mathbb{R}^{n})}\hookrightarrow L^{2^{\ast}}(\mathbb{R}^{n})$\ \ and $S_{s}$\ is attained (see \cite{CT}) in\ $\mathbb{R}^{n}$ by \ $\widetilde{u}_{\epsilon,
y}=\kappa(\varepsilon^{2}+|x-y|)^{-\frac{n-2s}{2}}$, where $\kappa\neq 0 \in \mathbb{R},\ \varepsilon>0 $  and $y\in \mathbb{R}^{n}$. That is

$$
S_{s}=\frac{\|\widetilde{u}_{\epsilon,y}||^{2}_{D_{s}(\mathbb{R}^{n})}}{(\int_{\mathbb{R}^{n}}|\widetilde{u}_{\epsilon,y}|^{2^{\ast}}dx)^{\frac{2}{2^{\ast}}}}.
$$

We normalize $\widetilde{u}_{\epsilon,y}$ as follow, let
$$
\overline{u}_{\epsilon,y}(x)=\frac{\widetilde{u}_{\epsilon,y}(x)}{||\widetilde{u}_{\epsilon,y}||_{2^{\ast}}}.
$$

By Lemma 2.12 in \cite{GLZ},
 $U_{\varepsilon, y}(x)= (S_{s})^{\frac{1}{2^{\ast}-2}}\overline{u}_{\epsilon,y}(x)$\ is a positive ground state solution of

\begin{equation}\label{int4}
(-\Delta)^{s}u= |u|^{2^{\ast}-2}u \ \ \text{in}\ \mathbb{R}^{n}
\end{equation}
and

\begin{equation}\label{int5}
||U_{\varepsilon, y}||^{2}_{D_{s}(\mathbb{R}^{n})}=\int_{\mathbb{R}^{n}}|U_{\varepsilon, y}|^{2^{\ast}}dx=(S_{s})^{\frac{n}{2s}}.
\end{equation}

Note that, the energy functional associated with \eqref{int1} is given by
$$
E(u,v)=\frac{1}{2}(||u||^{2}_{D_{s}(\mathbb{R}^{n})}+||v||^{2}_{D_{s}(\mathbb{R}^{n})})-\frac{1}{2^{\ast}}\int_{\mathbb{R}^{n}}(\mu_{1}|u|^{2^{\ast}}+\mu_{2}|v|^{2^{\ast}}+\gamma|u|^{\alpha}|v|^{\beta})dx.
$$

Define the Nehari manifold
\begin{align*}
\mathbb{N}=\{(u,v)&\in D_{s}(\mathbb{R}^{n})\times D_{s}(\mathbb{R}^{n}):u\neq 0,v\neq 0,\\
            &||u||^{2}_{D_{s}(\mathbb{R}^{n})}=\int_{\mathbb{R}^{n}}(\mu_{1}|u|^{2^{\ast}}+\frac{\alpha\gamma}{2^{\ast}}|u|^{\alpha}|v|^{\beta})dx,\\
            &||v||^{2}_{D_{s}(\mathbb{R}^{n})}=\int_{\mathbb{R}^{n}}(\mu_{2}|v|^{2^{\ast}}+\frac{\beta\gamma}{2^{\ast}}|u|^{\alpha}|v|^{\beta})dx\}.
\end{align*}

\begin{align*}
A:=\inf \limits_{(u,v)\in \mathbb{N}}E(u,v)&=\inf \limits_{(u,v)\in \mathbb{N}}\frac{s}{n}(||u||^{2}_{D_{s}(\mathbb{R}^{n})}+||v||^{2}_{D_{s}(\mathbb{R}^{n})})\\
                                          &=\inf \limits_{(u,v)\in \mathbb{N}}\frac{s}{n}\int_{\mathbb{R}^{n}}(\mu_{1}|u|^{2^{\ast}}+\mu_{2}|v|^{2^{\ast}}+\gamma|u|^{\alpha}|v|^{\beta})dx.\\
\end{align*}

Finally, we say that $(u,v)$ is a nontrivial solution of \eqref{int1} if $u\neq0,v\neq0$ and $(u,v)$ solves \eqref{int1}. Any nontrivial solution of \eqref{int1}\ is in $\mathbb{N}$. It is easy to see that when the
following algebra system \eqref{int6} has a solution $(k,l)$, then $(\sqrt{k}U_{\varepsilon, y},\sqrt{l}U_{\varepsilon, y})$ is a nontrivial solution of \eqref{int1}.

In this paper, we get the existence and nonexistence of least energy solutions of \eqref{int1} under certain conditions of $\gamma$, $n$ and $s$. Our existence results strongly depend on the following algebra system

\begin{equation}\label{int6}
\begin{cases}
\mu_{1}k^\frac{{2^{\ast}-2}}{2}+\frac{\alpha\gamma}{2^{\ast}}k^\frac{{\alpha-2}}{2}l^\frac{{\beta}}{2}= 1,\\

\mu_{2}l^\frac{{2^{\ast}-2}}{2}+\frac{\beta\gamma}{2^{\ast}}k^\frac{{\alpha}}{2}l^\frac{{\beta-2}}{2}=1,\\

k,l>0.
\end{cases}
\end{equation}

Our main results are:
\begin{theorem}\label{Th1}
 If $\gamma< 0,$\ \ then $A=\frac{s}{n}(\mu_{1}^{-\frac{n-2s}{2s}}+\mu_{2}^{-\frac{n-2s}{2s}})S^{\frac{n}{2s}}_{s}$\ and $A$ is not attained.
\end{theorem}

\begin{theorem}\label{Th2}
If $2s\leq n\leq 4s,\ \alpha>2,\ \beta >2$\ and
\begin{equation}\label{int7}
 0<\gamma \leq \frac{4ns}{(n-2s)^{2}}min\{\frac{\mu_{1}}{\alpha}(\frac{\alpha-2}{\beta-2})^{\frac{\beta-2}{2}},\frac{\mu_{2}}{\beta}(\frac{\alpha-2}{\beta-2})^{\frac{\alpha-2}{2}}\},
 \end{equation}
or\ $n>4s,\ 1<\alpha,\beta<2$\ and
\begin{equation}\label{int8}
\gamma\geq  \frac{4ns}{(n-2s)^{2}}max\{\frac{\mu_{1}}{\alpha}(\frac{2-\beta}{2-\alpha})^{\frac{2-\beta}{2}},\frac{\mu_{2}}{\beta}(\frac{2-\alpha}{2-\beta})^{\frac{2-\alpha}{2}}\},
\end{equation}
then\ $A=\frac{s}{n}(k_{0}+l_{0})S_{s}^{\frac{n}{2s}}$\ and $A$ is attained by $(\sqrt{k_{0}}U_{\varepsilon, y},\sqrt{l_{0}}U_{\varepsilon, y})$,\ where\ $(k_{0},l_{0})$\ satisfies \eqref{int6}\ and\  $k_{0}=min\{k : (k,l)\
satisfies \ \eqref{int6}\}$.
\end{theorem}

\begin{theorem}\label{Th3}
Assume\ $ n >4s $\ and\  $1<\alpha,\beta<2 $\ hold, there exists a
$$
\gamma_{1}\in(0,\frac{4ns}{(n-2s)^{2}}max\{\frac{\mu_{1}}{\alpha}(\frac{2-\beta}{2-\alpha})^{\frac{2-\beta}{2}},\frac{\mu_{2}}{\beta}(\frac{2-\alpha}{2-\beta})^{\frac{2-\alpha}{2}}],
$$
such that for any \ $\gamma\in(0,\gamma_{1})$,\ there exists a solution \ $(k(\gamma),l(\gamma))$\ of \eqref{int6}, satisfying $E(\sqrt{k(\gamma)}U_{\varepsilon, y},\sqrt{l(\gamma)}U_{\varepsilon, y})>A$ and\
$(\sqrt{k(\gamma)}U_{\varepsilon, y},\sqrt{l(\gamma)}U_{\varepsilon, y})$\ is a positive solution of \eqref{int1}.
\end{theorem}

\begin{remark}
Z. Guo, S. Luo and W. Zou \cite{GLZ} already showed the existence of positive least energy solutions for the system \eqref{int1} with $n>4s$ and $1<\alpha,\ \beta<2$ for all $\gamma>0$.  Theorem \ref{Th2} and Theorem
\ref{Th3} tell that if
$$
\gamma>=\gamma_0:\triangleq\frac{4ns}{(n-2s)^{2}}max\{\frac{\mu_{1}}{\alpha}(\frac{2-\beta}{2-\alpha})^{\frac{2-\beta}{2}},\frac{\mu_{2}}{\beta}(\frac{2-\alpha}{2-\beta})^{\frac{2-\alpha}{2}}\},
$$
the positive least energy solution of \eqref{int1} has to have the form\ $(\sqrt{k(\gamma)}U_{\varepsilon, y},\sqrt{l(\gamma)}U_{\varepsilon, y})$;
whereas, if $0<\gamma<\gamma_1$, \ $(\sqrt{k(\gamma)}U_{\varepsilon, y},\sqrt{l(\gamma)}U_{\varepsilon, y})$\ is not the least energy solution. However, it should be interesting to know whether $\gamma_1=\gamma_0$ or not.
If $\gamma_1\neq\gamma_0$, what happens when $\gamma\in [\gamma_1,\gamma_0]$.
\end{remark}

The paper is organized as follows. In section \ref{S2}, we introduce some preliminaries that will be used to prove theorems. In section \ref{S3}, we prove Theorem \ref{Th1}.  In section \ref{S4}, we prove Theorem \ref{Th2}.
The proof of Theorem \ref{Th3} is given in section \ref{S5}.

\section{Some Preliminaries }\label{S2}
For the case of $\gamma<0$, the following Lemma \ref{lem2} shows that if the energy functional attains its minimum at some point\ $(u,v)\in\mathbb{N}$,\ then\ $(u,v)$\ is a nontrivial solution of \eqref{int1}.

\begin{lemma}\label{lem2}
Assume $\gamma< 0,$ if $A$ is attained by a couple\ $(u,v)\in\mathbb{N},$ then\ (u,v)\  is a nontrivial solution of \eqref{int1}.
\end{lemma}

\begin{proof}
Define

\begin{align*}
 \mathbb{N}_{1}=\{&(u,v)\in D_{s}(\mathbb{R}^{n})\times D_{s}(\mathbb{R}^{n}):u\neq 0,v\neq  0\\
                 &G_{1}(u,v)=||u||^{2}_{D_{s}(\mathbb{R}^{n})}-\int_{\mathbb{R}^{n}}(\mu_{1}|u|^{2^{\ast}}+\frac{\alpha\gamma}{2^{\ast}}|u|^{\alpha}|v|^{\beta})dx=0\},\\
 \mathbb{N}_{2}=\{&(u,v)\in D_{s}(\mathbb{R}^{n})\times D_{s}(\mathbb{R}^{n}):u\neq 0,v\neq 0\\
                 &G_{2}(u,v)=||v||^{2}_{D_{s}(\mathbb{R}^{n})}-\int_{\mathbb{R}^{n}}(\mu_{2}|v|^{2^{\ast}}+\frac{\beta\gamma}{2^{\ast}}|u|^{\alpha}|v|^{\beta})dx=0\}.\\
\end{align*}

Then $\mathbb{N}=\mathbb{N}_{1}\bigcap \mathbb{N}_{2}.$ By the Fr\'{e}chet derivative, for any $\varphi_1,\varphi_2\in D_s({\mathbb R^n})$, we have
\begin{align*}
E'(u,v)(\varphi_{1},\varphi_{2})     &=(u,\varphi_{1})+(v,\varphi_{2})\\
                                     &-\int_{\mathbb{R}^{n}}(\mu_{1}|u|^{2^{\ast}-2}u\varphi_{1}+\mu_{2}|v|^{2^{\ast}-2}v\varphi_{2})dx\\
                                     &-\frac{\gamma}{2^{\ast}}\int_{\mathbb{R}^{n}}(\alpha|u|^{\alpha-2}u\varphi_{1}|v|^{\beta}+\beta|u|^{\alpha}|v|^{\beta-2}v\varphi_{2})dx,
\end{align*}

\begin{align*}
 G'_{1}(u,v)(\varphi_{1},\varphi_{2})&=2(u,\varphi_{1})-2^{\ast}\int_{\mathbb{R}^{n}}\mu_{1}|u|^{2^{\ast}-2}u\varphi_{1}dx\\
                                     &-\frac{\alpha\gamma}{2^{\ast}}\int_{\mathbb{R}^{n}}(\alpha|u|^{\alpha-2}u\varphi_{1}|v|^{\beta}+\beta|u|^{\alpha}|v|^{\beta-2}v\varphi_{2})dx,\\
 G'_{2}(u,v)(\varphi_{1},\varphi_{2})&=2(v,\varphi_{2})-2^{\ast}\int_{\mathbb{R}^{n}}\mu_{2}|v|^{2^{\ast}-2}v\varphi_{2}dx\\
                                     &-\frac{\beta\gamma}{2^{\ast}}\int_{\mathbb{R}^{n}}(\alpha|u|^{\alpha-2}u\varphi_{1}|v|^{\beta}+\beta|u|^{\alpha}|v|^{\beta-2}v\varphi_{2})dx,\\
\end{align*}
where

$$
(u,\varphi_{1})=\frac{C(n,\alpha)}{2}\int_{\mathbb{R}^{n}}\int_{\mathbb{R}^{n}}\frac{(u(x)-u(y))(\varphi_{1}(x)-\varphi_{1}(y))}{|y-x|^{n+2s}}dxdy.
$$

Then
$$
E'(u,v)(u,0)=G_{1}(u,v)=0,
$$

$$
E'(u,v)(0,v)=G_{2}(u,v)=0,
$$

\begin{align*}
G'_{1}(u,v)(u,0)&=2||u||^{2}_{D_{s}(\mathbb{R}^{n})}-2^{\ast}\int_{\mathbb{R}^{n}}\mu_{1}|u|^{2^{\ast}}dx-\frac{\alpha\gamma}{2^{\ast}}\int_{\mathbb{R}^{n}}\alpha|u|^{\alpha}|v|^{\beta}dx\\
                &=-(2^{\ast}-2)\int_{\mathbb{R}^{n}}\mu_{1}|u|^{2^{\ast}}dx+(2-\alpha)\int_{\mathbb{R}^{n}}\frac{\alpha\gamma}{2^{\ast}}|u|^{\alpha}|v|^{\beta}dx,\\
G'_{1}(u,v)(0,v)&=-\frac{\alpha\gamma}{2^{\ast}}\int_{\mathbb{R}^{n}}\beta|u|^{\alpha}|v|^{\beta}dx>0,\\
G'_{2}(u,v)(u,0)&=-\frac{\beta\gamma}{2^{\ast}}\int_{\mathbb{R}^{n}}\alpha|u|^{\alpha}|v|^{\beta}dx>0,\\
G'_{2}(u,v)(0,v)&=2||v||^{2}_{D_{s}(\mathbb{R}^{n})}-2^{\ast}\int_{\mathbb{R}^{n}}\mu_{2}|v|^{2^{\ast}}dx-\frac{\beta\gamma}{2^{\ast}}\int_{\mathbb{R}^{n}}\beta|u|^{\alpha}|v|^{\beta}dx\\
                &=-(2^{\ast}-2)\int_{\mathbb{R}^{n}}\mu_{2}|v|^{2^{\ast}}dx+(2-\beta)\int_{\mathbb{R}^{n}}\frac{\beta\gamma}{2^{\ast}}|u|^{\alpha}|v|^{\beta}dx.\\
\end{align*}

Suppose that $(u,v)\in \mathbb{N}$\ is a minimizer for $E$ restricted to $\mathbb{ N}$, then by the standard minimization theory, there exist two lagrange multipliers $L_{1},L_{2}\in \mathbb{R}$ such that,

$$
E'(u,v)+L_{1}G'_{1}(u,v)+L_{2}G'_{2}(u,v)=0.
$$

Then we have

$$
L_{1}G'_{1}(u,v)(u,0)+L_{2}G'_{2}(u,v)(u,0)=0,
$$

$$
L_{1}G'_{1}(u,v)(0,v)+L_{2}G'_{2}(u,v)(0,v)=0,
$$and

$$
G'_{1}(u,v)(u,0)+G'_{1}(u,v)(0,v)=-(2^{\ast}-2)||u||^{2}_{D_{s}(\mathbb{R}^{n})}\leq 0,
$$

$$
G'_{2}(u,v)(u,0)+G'_{2}(u,v)(0,v)=-(2^{\ast}-2)||v||^{2}_{D_{s}(\mathbb{R}^{n})}\leq 0.
$$

Since
$$
||u||^{2}_{D_{s}(\mathbb{R}^{n})}>0,\ ||v||^{2}_{D_{s}(\mathbb{R}^{n})}>0,
$$
hence
$$
|G'_{1}(u,v)(u,0)|=-G'_{1}(u,v)(u,0)>G'_{1}(u,v)(0,v),
$$

$$
|G'_{2}(u,v)(0,v)|=-G'_{2}(u,v)(0,v)>G'_{2}(u,v)(u,0).
$$

Define the matrix
\begin{equation*}M=       
\left(                 
  \begin{array}{cc}   
     G'_{1}(u,v)(u,0),& G'_{2}(u,v)(u,0)  \\  
    G'_{1}(u,v)(0,v), & G'_{2}(u,v)(0,v)  \\  
  \end{array}
\right),                 
\end{equation*}
then

$$
det(M)=|G'_{1}(u,v)(u,0)||G'_{2}(u,v)(0,v)|-G'_{1}(u,v)(0,v)G'_{2}(u,v)(u,0)>0,
$$which means\ $L_{1}=L_{2}=0$, that is $E'(u,v)=0.$\\
\end{proof}

Define functions
\begin{equation}\label{sec1}
\begin{cases}
F_{1}(k,l)= \mu_{1}k^\frac{{2^{\ast}-2}}{2}+\frac{\alpha\gamma}{2^{\ast}}k^\frac{{\alpha-2}}{2}l^\frac{{\beta}}{2}- 1, \ \ \  k>0,l\geq0,\\

F_{2}(k,l)= \mu_{2}l^\frac{{2^{\ast}-2}}{2}+\frac{\beta\gamma}{2^{\ast}}k^\frac{{\alpha}}{2}l^\frac{{\beta-2}}{2}-1, \ \ \ k\geq0,l>0,\\

l(k)=(\frac{2^{\ast}}{\alpha\gamma})^{\frac{2}{\beta}}k^{\frac{2-\alpha}{\beta}}(1-\mu_{1}k^{\frac{2^{\ast}-2}{2}})^{\frac{2}{\beta}},\ \ \ 0< k\leq\mu_{1}^{-\frac{2}{2^{\ast}-2}},\\

k(l)=(\frac{2^{\ast}}{\beta\gamma})^{\frac{2}{\alpha}}l^{\frac{2-\beta}{\alpha}}(1-\mu_{2}l^{\frac{2^{\ast}-2}{2}})^{\frac{2}{\alpha}},\ \ \ 0< l\leq\mu_{2}^{-\frac{2}{2^{\ast}-2}},
\end{cases}
\end{equation}
then

$$
F_{1}(k,l(k))\equiv0,\  F_{2}(k(l),l)\equiv0.
$$

\begin{remark}
If\ $(k,l)$\ satisfies \eqref{int6},\ then\ $(\sqrt{k}U_{\varepsilon, y},\sqrt{l}U_{\varepsilon, y})$\  is a nontrivial solution of\eqref{int1}, where $U_{\varepsilon, y}$\  satisfy \eqref{int4} and \eqref{int5}. Hence the
main work is to establish the existence of solutions to \eqref{int6}.
\end{remark}

In order to prove the existence results for  \eqref{int6}, we have the following Lemma \ref{lem3}.

\begin{lemma}\label{lem3}
Assume that $n>4s ,1<\alpha,\beta<2,\gamma>0$, then
\begin{equation}\label{sec2}
F_{1}(k,l)=0,\ F_{2}(k,l)=0,\quad k,l>0
\end{equation}
has a solution $(k_{0},l_{0})$ such that

\begin{equation}\label{sec3}
 F_{2}(k,l(k))<0,\quad \forall \hspace{0.1cm}  k\in(0,k_{0}),
 \end{equation}
where\ $(k_{0},l_{0})$\ satisfies $k_{0}=min\{k:(k,l) \ satisfies \ \eqref{int6}\}$. Similarly, \eqref{sec2} has a solution $(k_{1},l_{1})$, such that

\begin{equation}\label{sec4}
F_{1}(k(l),l)<0,\quad \forall \hspace{0.1cm} l\in(0,l_{1}).
\end{equation}
\end{lemma}

\begin{proof}
Solving $F_{1}(k,l)=0$ for $k,l>0,$\ we have \ $l=l(k)$ for all $ k\in (0,\mu_{1}^{-\frac{2}{2^{\ast}-2}}).$
Then, substituting this into $F_{2}(k,l)=0,$ we have

\begin{equation}\label{sec5}
\begin{array}{c}
\mu_{2}(\frac{2^{\ast}}{\alpha\gamma})^{\frac{\alpha}{\beta}}(1-\mu_{1}k^{\frac{2^{\ast}-2}{2}})^{\frac{\alpha}{\beta}}+\frac{\beta\gamma}{2^{\ast}}k^{-\frac{(2^{\ast}-2)\alpha}{2\beta}}\\
\\
-(\frac{2^{\ast}}{\alpha\gamma})^{\frac{2-\alpha}{\beta}}k^{-\frac{(2^{\ast}-2)(2-\beta)}{2\beta}}(1-\mu_{1}k^{\frac{2^{\ast}-2}{2}})^{\frac{2-\beta}{\beta}}=0.
\end{array}
\end{equation}

Let

\begin{align*}
f(k)=&\mu_{2}(\frac{2^{\ast}}{\alpha\gamma})^{\frac{\alpha}{\beta}}(1-\mu_{1}k^{\frac{2^{\ast}-2}{2}})^{\frac{\alpha}{\beta}}+\frac{\beta\gamma}{2^{\ast}}k^{-\frac{(2^{\ast}-2)\alpha}{2\beta}}\\
                                &-(\frac{2^{\ast}}{\alpha\gamma})^{\frac{2-\alpha}{\beta}}k^{-\frac{(2^{\ast}-2)(2-\beta)}{2\beta}}(1-\mu_{1}k^{\frac{2^{\ast}-2}{2}})^{\frac{2-\beta}{\beta}},
\end{align*}
then \eqref{sec5} has a solution is equivalent to $ f(k)=0$ \ has a solution in\ $(0,\mu_{1}^{-\frac{2}{2^{\ast}-2}})$. Since $1<\alpha,\beta<2$,\ we obtain

$$
\lim \limits _{k\rightarrow 0^{+}}f(k)=-\infty,\ f(\mu_{1}^{-\frac{2}{2^{\ast}-2}})=\frac{\beta\gamma}{2^{\ast}}\mu_{1}^{-\frac{\alpha}{\beta}}>0,
$$
then by the intermediate value theorem, there exists

$$
k_{0}\in(0,\mu_{1}^{-\frac{2}{2^{\ast}-2}})\ \ s.\  t. , f(k_{0})=0
$$ and

$$
f(k)<0, \ \forall \ k\in(0,k_{0}).
$$

Let\ $l_{0}=l(k_{0}),$\ then\ $(k_{0},l_{0})$\ is the solution of \eqref{sec2} and satisfies \eqref{sec3}.
Similarly, we can show \eqref{sec2} has a solution\ $(k_{1},l_{1})$,\ such that

\begin{equation*}
F_{1}(k(l),l)<0, \ \forall \ l\in(0,l_{1}).
\end{equation*}

\end{proof}

\begin{remark}\label{rem5}
From the proof of Lemma \ref{prop1} in the next a few pages, it is easy to see that system \eqref{int6} has only one solution $(k,l)=(k_{0},l_{0})$ under the assumption that $2s<n<4s,\ \alpha>2,\ \beta>2$\ and \eqref{int7}.
\end{remark}

\begin{remark}\label{rem2}
Obviously, if\ $(\sqrt{k_{0}}U_{\varepsilon, y},\sqrt{l_{0}}U_{\varepsilon, y})\in \mathbb{N}$, then

$$
A\leq E(\sqrt{k_{0}}U_{\varepsilon, y},\sqrt{l_{0}}U_{\varepsilon, y})=\frac{s}{n}(k_{0}+l_{0})S_{s}^{\frac{n}{2s}}.
$$
\end{remark}

Next, in order to show $A\geq \frac{s}{n}(k_{0}+l_{0})S_{s}^{\frac{n}{2s}}$, we require the following lemmas.\\

{\bf Case 1}. $n>4s,\ 1<\alpha,\beta<2$ and  \eqref{int8} hold.

\begin{lemma}\label{lem4}
Assume $n>4s ,\ 1<\alpha,\beta<2$\ and  \eqref{int8} hold, then

$$
l(k)+k \ is\  strictly\  increasing \  in\ [0,\mu_{1}^{-\frac{2}{2^{\ast}-2}}],
$$

$$
k(l)+l \ is \  strictly \ increasing \  in\  [0,\mu_{2}^{-\frac{2}{2^{\ast}-2}}].
$$
\end{lemma}

\begin{proof}
Since

\begin{align*}
l'(k)&=(\frac{2^{\ast}}{\alpha\gamma})^{\frac{2}{\beta}}\frac{2}{\beta}(k^{\frac{2-\alpha}{2}}-\mu_{1}k^{\frac{\beta}{2}})^{\frac{2-\beta}{\beta}}(\frac{2-\alpha}{2}k^{-\frac{\alpha}{2}}-\frac{\mu_{1}\beta}{2}k^{\frac{\beta-2}{2}})\\
     &=(\frac{2^{\ast}\mu_{1}}{\alpha\gamma})^{\frac{2}{\beta}}k^{\frac{2-2^{\ast}}{\beta}}(\mu_{1}^{-1}-k^{\frac{2^{\ast}-2}{2}})^{\frac{2-\beta}{\beta}}(\frac{2-\alpha}{\mu_{1}\beta}-k^{\frac{2^{\ast}-2}{2}}),\\
k'(l)&=(\frac{2^{\ast}}{\beta\gamma})^{\frac{2}{\alpha}}\frac{2}{\alpha}(l^{\frac{2-\beta}{2}}-\mu_{2}l^{\frac{\alpha}{2}})^{\frac{2-\alpha}{\alpha}}(\frac{2-\beta}{2}l^{-\frac{\beta}{2}}-\frac{\mu_{2}\alpha}{2}l^{\frac{\alpha-2}{2}})\\
     &=(\frac{2^{\ast}\mu_{2}}{\beta\gamma})^{\frac{2}{\alpha}}l^{\frac{2-2^{\ast}}{\alpha}}(\mu_{2}^{-1}-l^{\frac{2^{\ast}-2}{2}})^{\frac{2-\alpha}{\alpha}}(\frac{2-\beta}{\mu_{2}\alpha}-l^{\frac{2^{\ast}-2}{2}}),\\
\end{align*}
hence

$$
l'((\frac{2-\alpha}{\mu_{1}\beta})^{\frac{2}{2^{\ast}-2}})=l'(\mu_{1}^{-\frac{2}{2^{\ast}-2}})=0,
$$

$$
k'((\frac{2-\beta}{\mu_{2}\alpha})^{\frac{2}{2^{\ast}-2}})=k'(\mu_{2}^{-\frac{2}{2^{\ast}-2}})=0,
$$ and

\begin{align*}
 l'(k)>0&\Leftrightarrow k\in(0,(\frac{2-\alpha}{\mu_{1}\beta})^{\frac{2}{2^{\ast}-2}}),\\
 l'(k)<0&\Leftrightarrow k\in((\frac{2-\alpha}{\mu_{1}\beta})^{\frac{2}{2^{\ast}-2}},\mu_{1}^{-\frac{2}{2^{\ast}-2}}),\\
 k'(l)>0&\Leftrightarrow l\in(0,(\frac{2-\beta}{\mu_{2}\alpha})^{\frac{2}{2^{\ast}-2}}),\\
 k'(l)<0&\Leftrightarrow l\in((\frac{2-\beta}{\mu_{2}\alpha})^{\frac{2}{2^{\ast}-2}},\mu_{2}^{-\frac{2}{2^{\ast}-2}}).\\
\end{align*}

Next, we compute the second derivatives,

\begin{align*}
l''(k)&=\frac{2-\beta}{\beta}(\frac{2^{\ast}\mu_{1}}{\alpha\gamma})^{\frac{2}{\beta}}k^{\frac{2-2\beta-\alpha}{\beta}}(\mu_{1}^{-1}-k^{\frac{2^{\ast}-2}{2}})^{\frac{2-2\beta}{\beta}}\\
      &\times[(\frac{2-\alpha}{\mu_{1}\beta}-k^{\frac{2^{\ast}-2}{2}})^{2}-(\mu_{1}^{-1}-k^{\frac{2^{\ast}-2}{2}})(\frac{\alpha(2-\alpha)}{\mu_{1}\beta(2-\beta)}-k^{\frac{2^{\ast}-2}{2}})],\\
k''(l)&=\frac{2-\alpha}{\alpha}(\frac{2^{\ast}\mu_{2}}{\beta\gamma})^{\frac{2}{\alpha}}l^{\frac{2-2\alpha-\beta}{\alpha}}(\mu_{2}^{-1}-l^{\frac{2^{\ast}-2}{2}})^{\frac{2-2\alpha}{\alpha}}\\
      &\times[(\frac{2-\beta}{\mu_{2}\alpha}-l^{\frac{2^{\ast}-2}{2}})^{2}-(\mu_{2}^{-1}-l^{\frac{2^{\ast}-2}{2}})(\frac{\beta(2-\beta)}{\mu_{2}\alpha(2-\alpha)}-l^{\frac{2^{\ast}-2}{2}})],\\
\end{align*}we have

$$
l''(k)=0\Leftrightarrow k=(\frac{2(2-\alpha)}{\mu_{1}\beta(4-2^{\ast})})^{\frac{2}{2^{\ast}-2}}\triangleq k_{1}\in((\frac{2-\alpha}{\mu_{1}\beta})^{\frac{2}{2^{\ast}-2}},\mu_{1}^{-\frac{2}{2^{\ast}-2}}),
$$

$$
k''(l)=0\Leftrightarrow l=(\frac{2(2-\beta)}{\mu_{2}\alpha(4-2^{\ast})})^{\frac{2}{2^{\ast}-2}}\triangleq l_{1}\in((\frac{2-\beta}{\mu_{2}\alpha})^{\frac{2}{2^{\ast}-2}},\mu_{2}^{-\frac{2}{2^{\ast}-2}}).
$$

By  \eqref{int8}, we obtain
$$
l'(k)_{min}=l'(k_{1})=-(\frac{2^{\ast}(2^{\ast}-2)\mu_{1}}{2\alpha\gamma})^{\frac{2}{\beta}}(\frac{2-\beta}{2-\alpha})^{\frac{2-\beta}{\beta}}\geq -1,
$$

$$
k'(l)_{min}=k'(l_{1})=-(\frac{2^{\ast}(2^{\ast}-2)\mu_{2}}{2\beta\gamma})^{\frac{2}{\alpha}}(\frac{2-\alpha}{2-\beta})^{\frac{2-\alpha}{\alpha}}\geq -1.
$$

Hence
$$
l'(k)>-1,\ \forall \ k\in(0,\mu_{1}^{-\frac{2}{2^{\ast}-2}}),
$$

$$
k'(l)>-1,\ \forall\  l\in(0,\mu_{2}^{-\frac{2}{2^{\ast}-2}}),
$$
which means\  $l(k)+k$ \ is strictly increasing in\ $[0,\mu_{1}^{-\frac{2}{2^{\ast}-2}}] $\ and $k(l)+l$\ is strictly increasing in\ $[0,\mu_{1}^{-\frac{2}{2^{\ast}-2}}].$
\end{proof}

\begin{lemma}\label{lem5}
Assume $n>4s ,\ 1<\alpha,\beta<2$\ and  \eqref{int8} hold, \ $(k_{0},l_{0})$ is obtained in Lemma \ref{lem3}.\ Then

\begin{equation}\label{sec6}
(k_{0}+l_{0})^{\frac{2^{\ast}-2}{2}}max\{\mu_{1},\mu_{2}\}<1,
\end{equation}and

\begin{equation}\label{sec7}
F_{2}(k,l(k))<0,\ \  \forall\ k\in(0,k_{0});\ F_{1}(k(l),l)<0,\ \forall\ \ l\ \in(0,l_{0}).
\end{equation}

\end{lemma}

\begin{proof}

Since $k_{0}<\mu_{1}^{-\frac{2}{2^{\ast}-2}}.$ By Lemma \ref{lem4}, we have
$$
\mu_{1}^{-\frac{2}{2^{\ast}-2}}=l(\mu_{1}^{-\frac{2}{2^{\ast}-2}})+\mu_{1}^{-\frac{2}{2\ast-2}}\geq l(k_{0})+k_{0}=l_{0}+k_{0}.
$$

That is
$$
\mu_{1}(k_{0}+l_{0})^{\frac{2^{\ast}-2}{2}}\leq1.
$$

Similarly
$$
\mu_{2}(k_{0}+l_{0})^{\frac{2^{\ast}-2}{2}}\leq1.
$$

Hence
$$
(k_{0}+l_{0})^{\frac{2^{\ast}-2}{2}}max\{\mu_{1},\mu_{2}\}<1.
$$

To prove \eqref{sec7}, by Lemma \ref{lem3}, we only need to show that $(k_{0},l_{0})=(k_{1},l_{1}).$
By \eqref{sec3}, \eqref{sec4}, we have $l_{0}\geq l_{1},\ k_{1}\geq k_{0}$. Suppose by contradiction that $k_{1}> k_{0}，$,
then
$$
l(k_{1})+k_{1}>l(k_{0})+k_{0},
$$hence

$$
l_{1}+k(l_{1})=l(k_{1})+k_{1}>l(k_{0})+k_{0}=l_{0}+k(l_{0}).
$$

Since\ $k(l)+l$\ is strictly increasing for $l\in[0,\mu_{2}^{-\frac{2}{2^{\ast}-2}}],$
therefore $l_{1}>l_{0},$ which contradicts to $l_{0}\geq l$, then we have $k_{1}=k_{0}.$ Similarly, $l_{1}=l_{0}.$
\end{proof}

\begin{lemma}\label{prop2}
Assume\ $n>4s,1<\alpha,\beta<2$ \ and \eqref{int8} hold, then

\begin{equation}
\begin{cases}
k+l\leq k_{0}+l_{0},\\

F_{1}(k,l)\geq 0,\ F_{2}(k,l)\geq 0,\\

k,l>0,\ (k,l)\neq(0,0),
\end{cases}
\end{equation}
has an unique solution $(k,l)=(k_{0},l_{0}).$
\end{lemma}

\begin{proof}

Obviously $(k_{0},l_{0})$\ satisfies \eqref{prop2}. Suppose\ $(\widetilde{k},\ \widetilde{l})$ is another solution of\ \eqref{prop2}.
Without loss of generality, we may assume that\ $\widetilde{k}>0,$\ then $\widetilde{l}>0$. In fact, if\ $\widetilde{l}=0,$ then\ $ \widetilde{k}\leq k_{0}+l_{0}$
and
$$
F_{1}(\widetilde{k},0)=\mu_{1}\widetilde{k}^{\frac{2^{\ast}-2}{2}}-1\geq 0.
$$

Therefore

$$
1\leq\mu_{1}\widetilde{k}^{\frac{2^{\ast}-2}{2}}\leq\mu_{1}(k_{0}+l_{0})^{\frac{2^{\ast}-2}{2}},
$$
which contradicts to Lemma \ref{lem5}.\ In the following, we prove that $\widetilde{k}=k_{0}.$
Suppose by contradiction that $\widetilde{k}<k_{0},$ by the proof of Lemma \ref{lem4},\ we have $k(l)$\ is strictly increasing on\ $(0,(\frac{2-\beta}{\mu_{2}\alpha})^{\frac{2}{2^{\ast}-2}})$,\ and strictly decreasing on
$((\frac{2-\beta}{\mu_{2}\alpha})^{\frac{2}{2^{\ast}-2}},\mu_{2}^{-\frac{2}{2^{\ast}-2}})$.

On the one hand, since $k(0)=k(\mu_{2}^{-\frac{2}{2^{\ast}-2}})=0$ and
 $0<\widetilde{k}<k_{0},$ therefore, there exist $l_{1},\ l_{2 }$ satisfying $0<l_{1}<l_{2}<\mu_{2}^{-\frac{2}{2^{\ast}-2}},$ such that $k(l_{1})=k(l_{2})=\widetilde{k}$ and

\begin{equation}\label{prop3}
F_{2}({\widetilde{k}},l)<0\Leftrightarrow\widetilde{k}<k(l)\Leftrightarrow l_{1}<l<l_{2}.
\end{equation}

Furthermore, $F_{1}({\widetilde{k}},\ \widetilde{l})\geq 0,\ F_{2}({\widetilde{k}},\widetilde{l})\geq 0,$ we have $ \widetilde{l}>l(\widetilde{k}),\ \widetilde{l}\leq l_{1}\ or\ \widetilde{l}\geq l_{2}.$\\
By \eqref{sec7},\ we see， $F_{2}({\widetilde{k}},l(\widetilde{k}))<0 ,\ $ by \eqref{prop3},  we obtain \ $l_{1}<l<l_{2}$,
therefore $\widetilde{l}\geq l_{2}.$

On the other hand, \ let\ $l_{3}=k_{0}+l_{0}-\widetilde{k},$\ then\ $l_{3}>l_{0}$ and $$k(l_{3})+k_{0}+l_{0}-\widetilde{k}=k(l_{3})+l_{3}>k(l_{0})+l_{0}=k_{0}+l_{0},$$\ that is\ $k(l_{3})>\widetilde{k}.$ By \eqref{prop3},
we have $l_{1}<l_{3}<l_{2}.$\ Since $ \widetilde{k}+\widetilde{l} \leq k_{0}+l_{0},$
we obtain that $\widetilde{l}\leq k_{0}+l_{0}-\widetilde{k}=l_{3}<l_{2}.$
This contradicts to $\widetilde{l}\geq l_{2}, $ the proof completes.\\
\end{proof}

{\bf Case 2}. $2s<n<4s,\ \alpha>2,\ \beta>2$\ and \eqref{int7} hold.

\begin{lemma}\label{prop1}
Assume $c,d\in \mathbb{R}$ satisfy
\begin{equation}\label{PROP}
\begin{cases}
\mu_{1}k^\frac{{2^{\ast}-2}}{2}+\frac{\alpha\gamma}{2^{\ast}}k^\frac{{\alpha-2}}{2}l^\frac{{\beta}}{2}\geq 1,\\

\mu_{2}l^\frac{{2^{\ast}-2}}{2}+\frac{\beta\gamma}{2^{\ast}}k^\frac{{\alpha}}{2}l^\frac{{\beta-2}}{2}\geq1,\\

k,l>0.
\end{cases}
\end{equation}

If\ \  $2s<n<4s,\ \alpha,\beta>2$\ and \eqref{int7} hold, then\  $c+d\geq k+l,$ where $(k,l)\in {\mathbb R}^2 $ is the unique solution of \eqref{int6}.
\end{lemma}

\begin{proof}

Let\ $y=c+d,x=\frac{c}{d},y_{0}=k+l,x_{0}=\frac{k}{l}$,  by \eqref{PROP} and \eqref{int6} we have

\begin{align*}
y^{\frac{2^{\ast}-2}{2}}&\geq \frac{(x+1)^{\frac{2^{\ast}-2}{2}}}{\mu_{1}x^{\frac{2^{\ast}-2}{2}}+\frac{\alpha\gamma}{2^{\ast}}x^{\frac{\alpha-2}{2}}}\triangleq f_{1}(x),\ \ \ y_{0}^{\frac{2^{\ast}-2}{2}}=f_{1}(x_{0}),\\
y^{\frac{2^{\ast}-2}{2}}&\geq \frac{(x+1)^{\frac{2^{\ast}-2}{2}}}{\mu_{2}+\frac{\beta\gamma}{2^{\ast}}x^{\frac{\alpha}{2}}}\triangleq f_{2}(x),\ \ \ y_{0}^{\frac{2^{\ast}-2}{2}}= f_{2}(x_{0}).\\
\end{align*}

Then

\begin{align*}
   f'_{1}(x)&=\frac{\alpha\gamma(x+1)^{\frac{2^{\ast}-4}{2}}x^{\frac{\alpha-4}{2}}}{22^{\ast}(\mu_{1}x^{\frac{2^{\ast}-2}{2}}+\frac{\alpha\gamma}{2^{\ast}}x^{\frac{\alpha-2}{2}})^{2}}[-\frac{2^{\ast}(2^{\ast}-2)\mu_{1}}{\alpha\gamma}x^{\frac{\beta}{2}}+\beta
   x-(\alpha-2)],\\
   f'_{2}(x)&=\frac{\beta\gamma(x+1)^{\frac{2^{\ast}-4}{2}}}{22^{\ast}(\mu_{2}+\frac{\beta\gamma}{2^{\ast}}x^{\frac{\alpha}{2}})^{2}}[(\beta-2)x^{\frac{\alpha}{2}}-\alpha
   x^{\frac{\alpha-2}{2}}+\frac{2^{\ast}(2^{\ast}-2)\mu_{2}}{\beta\gamma}].\\
 \end{align*}

 Let

 \begin{align*}
    g_{1}(x)&=-\frac{2^{\ast}(2^{\ast}-2)\mu_{1}}{\alpha\gamma}x^{\frac{\beta}{2}}+\beta x-(\alpha-2),\\
    g_{2}(x)&=(\beta-2)x^{\frac{\alpha}{2}}-\alpha x^{\frac{\alpha-2}{2}}+\frac{2^{\ast}(2^{\ast}-2)\mu_{2}}{\beta\gamma},\\
\end{align*}then

\begin{align*}
        &g'_{1}(x)=0\Leftrightarrow  x=(\frac{2\alpha\gamma}{2^{\ast}(2^{\ast}-2)\mu_{1}})^{\frac{2}{\beta-2}}:= x_{1},\\
       &g'_{1}(x)>0\Leftrightarrow x<(\frac{2\alpha\gamma}{2^{\ast}(2^{\ast}-2)\mu_{1}})^{\frac{2}{\beta-2}},\\
       &g'_{1}(x)<0\Leftrightarrow x>(\frac{2\alpha\gamma}{2^{\ast}(2^{\ast}-2)\mu_{1}})^{\frac{2}{\beta-2}},\\
       &g'_{2}(x)=0\Leftrightarrow x=\frac{\alpha-2}{\beta-2}:= x_{2},\\
       &g'_{2}(x)>0\Leftrightarrow x>\frac{\alpha-2}{\beta-2},\ g'_{2}(x)<0\Leftrightarrow x<\frac{\alpha-2}{\beta-2}.
\end{align*}

Therefore, $g_{1}(x)$ increases in the interval $(0,(\frac{2\alpha\gamma}{2^{\ast}(2^{\ast}-2)\mu_{1}})^{\frac{2}{\beta-2}})$\ and decreases in the interval
$((\frac{2\alpha\gamma}{2^{\ast}(2^{\ast}-2)\mu_{1}})^{\frac{2}{\beta-2}},+\infty)$.
$g_{2}(x)$ decreases in the interval $(0,\frac{\alpha-2}{\beta-2})$ and increases in the interval $(\frac{\alpha-2}{\beta-2},+\infty).$

Hence

\begin{align*}
g_{1}(x)_{max}&=g_{1}(x_{1})=(\beta-2)(\frac{2\alpha\gamma}{2^{\ast}(2^{\ast}-2)\mu_{1}})^{\frac{2}{\beta-2}}-(\alpha-2),\\
g_{2}(x)_{min}&=g_{2}(x_{2})=-2(\frac{\alpha-2}{\beta-2})^{\frac{\alpha-2}{2}}+\frac{2^{\ast}(2^{\ast}-2)\mu_{2}}{\beta\gamma}.\\
\end{align*}

Then, by \eqref{int7},\ we have
$$
g_{1}(x)_{max}\leq 0,\ g_{2}(x)_{min}\geq 0.
$$

Therefore, $f_{1}(x)$\ is strictly decreasing in $(0,+\infty)$ and $f_{2}(x)$\ is strictly increasing in $(0,+\infty).$
Due to the fact that

$$
\lim \limits_{x\to 0}f_1(x)=+\infty,\quad \lim \limits_{x\to+\infty}f_2(x)=+\infty,
$$
there exists a unique \ $x_0>0$, such that\ $f_1(x_0)=f_2(x_0)$\ , which gives the uniqueness of\ $(k,l)$.

Since\ $f_1(x)\ge f_1(x_0)$\ for\ $x\le x_0$\ and\ $f_2(x)\ge f_2(x_0)$\ for\ $x\ge x_0$, we get

\begin{align*}
y^{\frac{2^{\ast}-2}{2}}&\geq max\{f_{1}(x),{f_2(x)}\},\\
                        &=f_1(x_0)=y_{0}^{\frac{2^{\ast}-2}{2}},
\end{align*}which means $c+d\geq k+l.$
\end{proof}

\section{Proof of Theorem 1.1}\label{S3}

\begin{proof}
By Lemma \ref{lem2}, we obtain that when $-\infty<\gamma< 0$ and if $A$ is attained by a couple\ $(u,v)\in\mathbb{N},$ then\ $(u,v)$ \  is a solution of \eqref{int1}.
For any  $(u,v)\in \mathbb{N},$

\begin{align*}
 ||u||^{2}_{D_{s}(\mathbb{R}^{n})}&=\int_{\mathbb{R}^{n}}\mu_{1}|u|^{2^{\ast}}+\frac{\alpha\gamma}{2^{\ast}}|u|^{\alpha}|v|^{\beta}dx\\
                         &\leq\int_{\mathbb{R}^{n}}\mu_{1}|u|^{2^{\ast}}dx\leq \mu_{1}S_{s}^{-\frac{2^{\ast}}{2}}(||u||^{2}_{D_{s}(\mathbb{R}^{n})})^{\frac{2^{\ast}}{2}}.
\end{align*}

Therefore

$$
||u||^{2}_{D_{s}(\mathbb{R}^{n})}\geq\mu_{1}^{-\frac{n-2s}{2s}}S^{\frac{n}{2s}}_{s}.
$$

Similarly

$$
||v||^{2}_{D_{s}(\mathbb{R}^{n})}\geq\mu_{2}^{-\frac{n-2s}{2s}}S^{\frac{n}{2s}}_{s}.
$$

Hence
\begin{equation*}
A\geq\frac{s}{n}(||u||^{2}_{D_{s}(\mathbb{R}^{n})}+||v||^{2}_{D_{s}(\mathbb{R}^{n})}\geq\frac{s}{n}(\mu_{1}^{-\frac{n-2s}{2s}}+\mu_{2}^{-\frac{n-2s}{2s}})S^{\frac{n}{2s}}_{s}.
\end{equation*}

By lemma 2.12\ in \cite{GLZ},\ we know\ $w_{\mu_{i}}=(\frac{S_{S}}{\mu_{i}})^{\frac{1}{2^{\ast}-2}}\overline{u}_{\epsilon,y}=(\frac{1}{\mu_{i}})^{\frac{1}{2^{\ast}-2}}U_{\epsilon,y}$
is the solution of the equation

$$
(-\Delta)^{s}u= \mu_{i}|u|^{2^{\ast}-2}u, \ \ \ \   \text{in} \ \ \mathbb{R}^{n}
$$for $i= 1,2.$

Let\ $e_{1}=(1,0,\cdots 0)\in \mathbb{R}^{n},\ (U(X),V_{R}(X))=(w_{\mu_{1}}(x),\ w_{\mu_{2}}(x+Re_{1}),$
where R is a positive constant. Since \ $V_{R}(x)\in D_s(\mathbb{R}^{n})$ is a solution of

$$
(-\Delta)^{s}u= \mu_{2}|u|^{2^{\ast}-2}u \ \ \ \ \text{in} \ \ \mathbb{R}^{n},
$$
we have, $V_{R}(X)\rightharpoonup0 \ in \ L^{2^{\ast}}(\mathbb{R}^{n})$\ as $R\rightarrow +\infty,$\ hence

\begin{align*}
 \lim \limits _{R\rightarrow +\infty}\int_{\mathbb{R}^{n}}U^{\alpha}V_{R}^{\beta}dx
        &=  \lim \limits _{R\rightarrow +\infty}\int_{\mathbb{R}^{n}}U^{\alpha}V_{R}^{\frac{\alpha}{2^{\ast}-1}}V_{R}^{\frac{2^{\ast}(\beta-1)}{2^{\ast}-1}}dx,\\
        &\leq \lim \limits _{R\rightarrow +\infty}(\int_{\mathbb{R}^{n}}U^{2^{\ast}-1}V_{R}dx)^{\frac{\alpha}{2^{\ast}-1}}(\int_{R^{n}}V_{R}^{2^{\ast}}dx)^{\frac{\beta-1}{2^{\ast}-1}}\rightarrow 0.
\end{align*}

We will show that for $R>0$ sufficiently large, the system

\begin{equation}\label{pr1}
\begin{cases}
||U||^{2}_{D_{s}(\mathbb{R}^{n})}=\int_{\mathbb{R}^{n}}\mu_{1}|U|^{2^{\ast}}dx \\  \ \ \ \ \ \ \ \ \ \ \ \ \ \ \ \ = (t_{R})^{\frac{2^{\ast}-2}{2}}\int_{\mathbb{R}^{n}}\mu_{1}|U|^{2^{\ast}}dx
+(t_{R})^{\frac{\alpha-2}{2}}(s_{R})^{\frac{\beta}{2}}\int_{\mathbb{R}^{n}}\frac{\alpha\gamma}{2^{\ast}}U^{\alpha}V_{R}^{\beta}dx,\\ ||V_{R}||^{2}_{D_{s}(\mathbb{R}^{n})}=\int_{\mathbb{R}^{n}}\mu_{2}|V_{R}|^{2^{\ast}}dx\\\
\ \ \ \ \ \ \ \ \ \ \ \ \ \ \
=(s_{R})^{\frac{2^{\ast}-2}{2}}\int_{\mathbb{R}^{n}}\mu_{2}|V_{R}|^{2^{\ast}}dx
+(t_{R})^{\frac{\alpha}{2}}(s_{R})^{\frac{\beta-2}{2}}\int_{\mathbb{R}^{n}}\frac{\alpha\gamma}{2^{\ast}}U^{\alpha}V_{R}^{\beta}dx,\\
\end{cases}
\end{equation}
has a solution\ $(t_{R},s_{R})$ with
\[
\lim \limits _{R\rightarrow +\infty}(|t_{R}-1|+|s_{R}-1|)=0\ \ \ \ \ \ \ \ \ \ \ \ \ \ \ (\star)
\]
which implies that

$$
(\sqrt{t_{R}}U,\sqrt{S_{R}}V_{R})\in \mathbb{N}.
$$

Let us assume $(\star)$ first, then
\begin{align*}
       A&=\inf \limits_{(u,v)\in \mathbb{N}}E(u,v)\leq E(\sqrt{t_{R}}U,\sqrt{S_{R}}V_{R}),\\
        &=\frac{s}{n}(t_{R}||U||^{2}_{D_{s}(\mathbb{R}^{n})})+s_{R}||V_{R}||^{2}_{D_{s}(\mathbb{R}^{n})}),\\
        &\leq \frac{s}{n}(t_{R}\mu_{1}^{-\frac{n-2s}{2s}}||U_{\varepsilon, y}||^{2}_{D_{s}(\mathbb{R}^{n})}+s_{R}\mu_{2}^{-\frac{n-2s}{2s}}||U_{\varepsilon ,y}||^{2}_{D_{s}(\mathbb{R}^{n})}).\\
        \end{align*}

Let\ $R\rightarrow +\infty $, we get

$$
A \leq \frac{s}{n}(\mu_{1}^{-\frac{n-2s}{2s}}+\mu_{2}^{-\frac{n-2s}{2s}})S^{\frac{n}{2s}}_{s}.
$$

Therefore

\begin{equation}\label{pr2}
A=\frac{s}{n}(\mu_{1}^{-\frac{n-2s}{2s}}+\mu_{2}^{-\frac{n-2s}{2s}})S^{\frac{n}{2s}}_{s}.
\end{equation}

Suppose that $A$ is attained by some\ $(u,v)\in \mathbb{N} $, then $E(u,v)=A$. By Lemma \ref{lem2}, we know\ $(u,v)$\ is a nontrivial solution of \eqref{int1}.\ By Strong maximum principle for fractional Laplacian( see,
Proposition 2.17 in \cite{LS}, Lemma 6 in \cite{SV2}) and comparison principle in \cite{XR}, we may assume that $u>0,\ v>0$\ and

$$
\int_{\mathbb{R}^{n}}|u|^{\alpha}|v|^{\beta}dx>0,
$$then

$$
||u||^{2}_{D_{s}(\mathbb{R}^{n})}<\int_{\mathbb{R}^{n}}\mu_{1}|u|^{2^{\ast}}dx\leq \mu_{1}S_{s}^{-\frac{2^{\ast}}{2}}(||u||^{2}_{D_{s}(\mathbb{R}^{n})})^{\frac{2^{\ast}}{2}}.
$$

Hence

$$
||u||^{2}_{D_{s}(\mathbb{R}^{n})}>\mu_{1}^{-\frac{n-2s}{2s}}S^{\frac{n}{2s}}_{s}.
$$

Similarly, we have

$$
||v||^{2}_{D_{s}(\mathbb{R}^{n})}>\mu_{2}^{-\frac{n-2s}{2s}}S^{\frac{n}{2s}}_{s}.
$$

Then

$$
A=E(u,v)=\frac{s}{n}(||u||^{2}_{D_{s}(\mathbb{R}^{n})}+||v||^{2}_{D_{s}(\mathbb{R}^{n})})>\frac{s}{n}(\mu_{1}^{-\frac{n-2s}{2s}}+\mu_{2}^{-\frac{n-2s}{2s}})S^{\frac{n}{2s}}_{s},
$$which contradicts to \eqref{pr2}.  Therefore, $A$ is not be obtained.

Now, we claim $(\star)$,

{\bf Proof of $(\star)$.}  Let

$$
\theta:\triangleq\frac{\int_{\mathbb{R}^{n}}U^{\alpha}V_{R}^{\beta}dx}{\int_{\mathbb{R}^{n}}\mu_{1}|U|^{2^{\ast}}dx},
$$
we have\ $\theta\rightarrow0$  as R sufficiently large.\ Then \eqref{pr1} has a solution is equivalent to that the following system has a solution at the neighbourhood of point $(1,1)$.

\begin{equation*}
\begin{cases}
k^\frac{{2^{\ast}-2}}{2}+\frac{\alpha\gamma}{2^{\ast}}k^\frac{{\alpha-2}}{2}l^\frac{{\beta}}{2}\theta=1,\\

l^\frac{{2^{\ast}-2}}{2}+\frac{\beta\gamma}{2^{\ast}}k^\frac{{\alpha}}{2}l^\frac{{\beta-2}}{2}\theta=1.
\end{cases}
\end{equation*}

By Taylor expansion at $(1,1)$, we have
\begin{equation*}
\begin{cases}
(1+\frac{{2^{\ast}-2}}{2}(k-1)+O((k-1)^{2}))\\
+\frac{\alpha\gamma}{2^{\star}}(1+\frac{{\alpha-2}}{2}(k-1)+O((k-1)^{2}))(1+\frac{{\beta}}{2}(l-1)+O(l-1)^{2})\theta=1,\\

(1+\frac{{2^{\ast}-2}}{2}(l-1)+O((l-1)^{2}))\\
+\frac{\beta\gamma}{2^{\star}}(1+\frac{{\alpha}}{2}(k-1)+O((k-1)^{2}))(1+\frac{{\beta}-2}{2}(l-1)+O((l-1)^{2}))\theta=1.
\end{cases}
\end{equation*}

Therefore

\begin{equation*}
\begin{cases}
k-1=-\frac{\alpha\gamma}{2^{\ast}}(\frac{\alpha-2}{2^{\ast}-2})(k-1)\theta-\frac{\alpha\gamma}{2^{\ast}}(\frac{\beta}{2^{\ast}-2})(l-1)\theta-\frac{\alpha\gamma}{2^{\ast}}(\frac{2}{2^{\ast}-2})\theta\\\ \ \ \ \ \ \ \ \
+O((k-1)^{2}+(l-1)^{2}),\\

l-1=-\frac{\beta\gamma}{2^{\ast}}(\frac{\alpha}{2^{\ast}-2})(k-1)\theta-\frac{\beta\gamma}{2^{\ast}}(\frac{\beta-2}{2^{\ast}-2})(l-1)\theta-\frac{\beta\gamma}{2^{\ast}}(\frac{2}{2^{\ast}-2})\theta\\ \ \ \ \ \ \ \ \ \
+O((k-1)^{2}+(l-1)^{2}).
\end{cases}
\end{equation*}

Let

$$
a_{1}=-\frac{\alpha\gamma}{2^{\ast}}(\frac{\alpha-2}{2^{\ast}-2}),b_{1}=-\frac{\alpha\gamma}{2^{\ast}}(\frac{\beta}{2^{\ast}-2}),c_{1}=-\frac{\alpha\gamma}{2^{\ast}}(\frac{2}{2^{\ast}-2}),
$$

$$
a_{2}=-\frac{\beta\gamma}{2^{\ast}}(\frac{\alpha}{2^{\ast}-2}),b_{2}=-\frac{\beta\gamma}{2^{\ast}}(\frac{\beta-2}{2^{\ast}-2}),c_{2}=-\frac{\beta\gamma}{2^{\ast}}(\frac{2}{2^{\ast}-2}),
$$we have that

\begin{equation}\label{pr3}
\begin{cases}
k-1=\left(a_{1}(k-1)+b_{1}(l-1)+c_{1}\right)\theta+O((k-1)^{2}+(l-1)^{2}),\\

l-1=\left(a_{2}(k-1)+b_{2}(l-1)+c_{2}\right)\theta+O((k-1)^{2}+(l-1)^{2}).
\end{cases}
\end{equation}

Let

\begin{equation*}
x=\left(
  \begin{array}{c}
     k-1  \\
    l-1  \\
  \end{array}
\right),
\quad   B=\left(
  \begin{array}{cc}
     a_{1}& b_{1}  \\
     a_{2}  &  b_{2}  \\
  \end{array}
\right),
\quad c=\left(
  \begin{array}{c}
     c_{1} \\
     c_{2}\\
  \end{array}
\right),
\end{equation*}
then

$$
x=\theta Bx+c\theta +O(x^2).
$$
Let$$T(x)\triangleq \theta Bx+c\theta +O(x^2),$$
then
$$
\|T(x)\|\le \theta \|B\|\|x\|+\|c\|\theta+O(\|x\|^2).
$$

If we choose $x\in {\mathbb R^2}$ with $\|x\|\le 2\|c\|\theta$, we get
$$
\|T(x)\|\le (2\|B\|\|c\|\theta+\|c\|+O(\theta))\theta\le 2\|c\|\theta,
$$when $\theta$ is small enough. Since $T$ is continuous and $\theta\rightarrow 0,$ as $R\rightarrow\infty$, by Brouwer's fixed point theorem, we get
that the system \eqref{pr1} has a solution $(t_{R},s_{R})$ for all large $R$ \ with

$$
\lim \limits _{R\rightarrow +\infty}(|t_{R}-1|+|s_{R}-1|)=0.
$$
\end{proof}

\section{Proof of Theorem 1.2}\label{S4}

\begin{proof}

By Remark \ref{rem2}, we have
\begin{equation}
A\leq E(\sqrt{k_{0}}U_{\varepsilon, y},\sqrt{l_{0}}U_{\varepsilon, y})=\frac{s}{n}(k_{0}+l_{0})S_{s}^{\frac{n}{2s}}.
\end{equation}

Let\ $(u_{i},\ v_{i})\in \mathbb{N}$ be a minimizing sequence for $A$,
that is $E(u_{i},v_{i})\rightarrow A$ as\ $n\rightarrow\infty.$

Define
$$
c_{i}=(\int_{\mathbb{R}^{n}}|u_{i}|^{2\ast}dx)^{\frac{2}{2\ast}},\ d_{i}=(\int_{\mathbb{R}^{n}}|v_{i}|^{2\ast}dx)^{\frac{2}{2\ast}},
$$then

\begin{align*}
S_{s}c_{i}&\leq ||u_{i}||^{2}_{D_{s}(\mathbb{R}^{n})}=\int_{\mathbb{R}^{n}}\mu_{1}|u_{i}|^{2\ast}+\frac{\alpha\gamma}{2^{\ast}}|u_{i}|^{\alpha}|v_{i}|^{\beta}dx\\
          &\leq \mu_{1}(c_{i})^{\frac{2^{\ast}}{2}}+\frac{\alpha\gamma}{2^{\ast}}(\int_{\mathbb{R}^{n}}|u_{i}|^{2^{\ast}})^{\frac{\alpha}{2^{\ast}}}(\int_{\mathbb{R}^{n}}|v_{i}|^{2^{\ast}})^{\frac{\beta}{2^{\ast}}}\\
          &\leq\mu_{1}(c_{i})^{\frac{2^{\ast}}{2}}\frac{\alpha\gamma}{2^{\ast}}(c_{i})^{\frac{\alpha}{2}}(d_{i})^{\frac{\beta}{2}},
\end{align*}

\begin{align*}
S_{s}d_{i}&\leq ||v_{i}||^{2}_{D_{s}(\mathbb{R}^{n})}=\int_{\mathbb{R}^{n}}\mu_{1}|v_{i}|^{2\ast}+\frac{\alpha\gamma}{2^{\ast}}|u_{i}|^{\alpha}|v_{i}|^{\beta}dx\\
          &\leq \mu_{2}(d_{i})^{\frac{2^{\ast}}{2}}\frac{\beta\gamma}{2^{\ast}}(c_{i})^{\frac{\alpha}{2}}(d_{i})^{\frac{\beta}{2}}.
\end{align*}

Dividing both side of inequality by $S_{s}c_{i}$\ and\ $S_{s}d_{i}$. Let

$$
\widetilde{c_{i}}=\frac{c_{i}}{S_{s}^{\frac{2}{2^{\ast-2}}}},\quad\quad \widetilde {d_{i}}=\frac{d_{i}}{S_{s}^{\frac{2}{2^{\ast-2}}}},
$$we get

\begin{equation*}
\begin{cases} \mu_{1}\widetilde{c_{i}}^\frac{{2^{\ast}-2}}{2}+\frac{\alpha\gamma}{2^{\star}}\widetilde{c_{i}}^\frac{{\alpha-2}}{2}\widetilde{d_{i}}^\frac{{\beta}}{2}\geq 1,\\

\mu_{2}\widetilde{c_{i}}^\frac{{2^{\ast}-2}}{2}+\frac{\beta\gamma}{2^{\star}}\widetilde{c_{i}}^\frac{{\alpha}}{2}\widetilde{d_{i}}^\frac{{\beta-2}}{2}\geq1,
\end{cases}
\end{equation*}
that is

$$
F_{1}(\widetilde{c_{i}},\widetilde{d_{i}})\geq0,\ F_{2}(\widetilde{c_{i}},\widetilde{d_{i}})\geq0.
$$

Consequently, for the case \ $2s<n<4s,\ \alpha>2,\ \beta>2$ and  \eqref{int7} hold, Lemma \ref{prop1} ensures that
$$
\widetilde{c_{i}}+\widetilde{d_{i}}\geq k+l=k_{0}+l_{0}.
$$

For the case  $n>4s ,1<\alpha,\beta<2$\ and  \eqref{int8} hold, Lemma \ref{prop2} ensures
$$
\widetilde{c_{i}}+\widetilde{d_{i}}\geq k+l=k_{0}+l_{0}.
$$

Therefore
$$
c_{i}+d_{i}\geq (k_{0}+l_{0})S_{s}^{\frac{2}{2^{\ast}-2}}=(k_{0}+l_{0})S_{s}^{\frac{n-2s}{2s}}.
$$

Since\ $(u_{i},v_{i})\in \mathbb{N}$ is a minimizing sequence for $A$, we have

\begin{equation}\label{gt1}
 E(u_{i},v_{i})=\frac{s}{n}(||u_{i}||^{2}_{D_{s}(\mathbb{R}^{n})}+||v_{i}||^{2}_{D_{s}(\mathbb{R}^{n})}).
\end{equation}

By the definition of $S_{s}$ and \eqref{gt1}, we have,

$$
S_{s}(c_{i}+d_{i})\leq(||u_{i}||^{2}_{D_{s}(\mathbb{R}^{n})}+||v_{i}||^{2}_{D_{s}(\mathbb{R}^{n})}) =\frac{n}{s}E(u_{i},v_{i})=\frac{n}{s}A+o(1).
$$

Since
$$
S_{s}(c_{i}+d_{i})\geq(k_{0}+l_{0})S_{s}^{\frac{n}{2s}},\ A\leq\frac{s}{n}(k_{0}+l_{0})S_{s}^{\frac{n}{2s}},
$$we have

$$
(k_{0}+l_{0})S_{s}^{\frac{n}{2s}}\leq S_{s}(c_{i}+d_{i})\leq(k_{0}+l_{0})S_{s}^{\frac{n}{2s}},
$$that is

$$
c_{i}+d_{i}\rightarrow(k_{0}+l_{0})S_{s}^{\frac{2}{2^{\ast}-2}}\  as\ i\rightarrow+\infty,
$$thus

$$
A=\lim \limits _{i\rightarrow +\infty}E(u_{i},v_{i})\geq\lim \limits _{i\rightarrow +\infty}\frac{s}{n}S_{s}(c_{i}+d_{i})\geq\frac{s}{n}(k_{0}+l_{0})S_{s}^{\frac{n}{2s}}.
$$

Hence
$$
A=\frac{s}{n}(k_{0}+l_{0})S_{s}^{\frac{n}{2s}}=E(\sqrt{k_{0}}U_{\varepsilon, y},\sqrt{l_{0}}U_{\varepsilon, y}).
$$
\end{proof}

\section{Proof of Theorem 1.3}\label{S5}
In order to proof Theorem \ref{Th3}, we use a result of Z. Guo, S. Luo and W. Zou in \cite{GLZ}.

\begin{theorem}[Theorem 1.1 of \cite{GLZ}]\label{Th5}
Assume $(H)$ holds, where
\begin{equation*}
(H)=\left\{\begin{array}{ll}
    1<\alpha,\ \beta<2, &\mbox{ if } 4s<N<6s;\\
    \\
    \alpha,\ \beta>1,& \mbox{ if } N\ge 6s.
    \end{array}\right.
\end{equation*}

Then \eqref{int1} has a positive ground state solution $(U,\ V)$ for all $\gamma>0$, which is radially symmetric decreasing with the following decay condition
$$
U(x),\  V(x)\le C(1+|x|)^{2s-n}.
$$

That is $E(U,V)=A,$ where
$$
\begin{array}{c}
A=\inf \limits_{(u,v)\in {\mathcal N}} E(u,v),\\
\\
{\mathcal N}=\left\{\begin{array}{c}
    (u,v)\in D_s({\mathbb R}^n)\times D_s({\mathbb R}^n)\backslash \{(0,0)\}:\|u\|^2_{D_s({\mathbb R}^n)}+\|u\|^2_{D_s({\mathbb R}^n)}=\\
    \\
    \int_{{\mathbb R}^n}(\mu_1|u|^{2*}+\mu_2|v|^{2^*}+\gamma|u|^{\alpha}|v|^{\beta})
    \end{array}\right\}.
\end{array}
$$
\end{theorem}

\begin{lemma}[A result after Lemma 3.1 in \cite{GLZ}]\label{lem6}
 Let $u_{\mu}=(\frac{S_{S}}{\mu})^{\frac{1}{2^{\ast}-2}}\overline{u}$ is a positive ground state solution of
$$
(-\Delta)^{s}u= \mu|u|^{2^{\ast}-2}u, \ \text{in}\ \ \mathbb{R}^{n}
$$ then

\begin{align*}
&A<min\{\inf \limits_{(u,0)\in \mathcal{N}}E(u,0),\inf \limits_{(0,v)\in\mathcal{ N}}E(0,v)\}\\
& =min\{E(u_{\mu_{1}},0),E(0,u_{\mu_{2}})\}\\
&=min\{\frac{s}{n}\mu_{1}^{-\frac{n-2s}{2s}}S_{s}^{\frac{n}{2s}},\frac{s}{n}\mu_{2}^{-\frac{n-2s}{2s}}S_{s}^{\frac{n}{2s}}\}
\end{align*}
\end{lemma}

\begin{remark}\label{rem6}
The Nahri manifold used in \cite{GLZ} is different from the Nahri manifold used in our paper. However, the positive solutions of \eqref{int1} are certainly in both of the Nahri manifolds, therefore, Z. Guo, S. Luo and W.
Zou's result: Theorem \ref{Th5} ensures that
$$
A=\inf \limits_{(u,v)\in {\mathcal N}} E(u,v)=\inf \limits_{(u,v)\in {\mathbb N}} E(u,v).
$$

\end{remark}
\begin{proof}[Proof of Theorem \ref{Th3}]  To obtain the existence of $(k(\gamma),\ l(\gamma)$\ for $\gamma >0$,
 we define functions
$$
F_{1}(k,l,\gamma)= \mu_{1}k^\frac{{2^{\ast}-2}}{2}+\frac{\alpha\gamma}{2^{\ast}}k^\frac{{\alpha-2}}{2}l^\frac{{\beta}}{2}- 1 \ \ k,l>0,
$$

$$
F_{2}(k,l,\gamma)= \mu_{2}l^\frac{{2\ast-2}}{2}+\frac{\beta\gamma}{2^{\star}}k^\frac{{\alpha}}{2}l^\frac{{\beta-2}}{2}-1\ \ k,l>0,
$$let

$$
k(0)=\mu_{1}^{-\frac{2}{2^{\ast}-2}},\ l(0)=\mu_{2}^{-\frac{2}{2^{\ast}-2}},
$$then

\begin{align*}
             F_{1}(k(0),l(0),0)&=F_{2}(k(0),l(0),0)=0,\\
 \partial_{k}F_{1}(k(0),l(0),0)&=\frac{2^{\ast}-2}{2}\mu_{1}k^{\frac{2^{\ast}-4}{2}}>0,\\
 \partial_{l}F_{1}(k(0),l(0),0)&=\partial_{k}F_{2}(k(0),l(0),0)=0,\\
 \partial_{l}F_{2}(k(0),l(0),0)&=\frac{2^{\ast}-2}{2}\mu_{2}l^{\frac{2^{\ast}-4}{2}}>0.\\
\end{align*}

Denote
\begin{equation*}D=
\left(
  \begin{array}{cc}
     \partial_{k}F_{1}(k(0),l(0),0),& \partial_{l}F_{1}(k(0),l(0),0)  \\
    \partial_{k}F_{2}(k(0),l(0),0), & \partial_{l}F_{2}(k(0),l(0),0) \\
  \end{array}
\right),
\end{equation*}
since $det(D)>0$, by implicit function theorem,\ we see that\ $k(\gamma),\ l(\gamma)$ are well defined
and of class \ $C^{1}$ in\ $(-\gamma_{2},\gamma_{2})$,\ for some $\gamma_{2}>0$ and

$$
F_{1}(k(\gamma),\ l(\gamma),\ \gamma)=F_{2}(k(\gamma),\ l(\gamma),\ \gamma)=0.
$$

Thus\ $(\sqrt{k_{\gamma}}U_{\varepsilon, y},\sqrt{l_{\gamma}}U_{\varepsilon, y})$ is a solution of \eqref{int1}.

Since
$$
\lim \limits _{\gamma\rightarrow 0}(k(\gamma)+l(\gamma))=k(0)+l(0)=\mu_{1}^{-\frac{n-2s}{2s}}+\mu_{2}^{-\frac{n-2s}{2s}},
$$there exists a $\gamma_{1}\in(0,\gamma_{2})$, such that

$$
k(\gamma)+l(\gamma)>min\{\mu_{1}^{-\frac{n-2s}{2s}},\mu_{2}^{-\frac{n-2s}{2s}}\},\ \forall\ \gamma\in(0,\gamma_{1}).
$$

By Lemma \ref{lem6} and Remark \ref{rem6}, we get
\begin{align*}
E(\sqrt{k_{\gamma}}U_{\varepsilon, y},\sqrt{l_{\gamma}}U_{\varepsilon, y})&=\frac{s}{n}(k(\gamma)+l(\gamma))S_{s}^{\frac{n}{2s}}\\
    &>min\{\frac{s}{n}\mu_{1}^{-\frac{n-2s}{2s}}S_{s}^{\frac{n}{2s}},\frac{s}{n}\mu_{2}^{-\frac{n-2s}{2s}}S_{s}^{\frac{n}{2s}}\}\\
    &>A=\inf \limits_{(u,v)\in {\mathbb N}}E(u,v),
\end{align*}
that is $(\sqrt{k_{\gamma}}U_{\varepsilon, y},\sqrt{l_{\gamma}}U_{\varepsilon, y})$ is another positive solution of \eqref{int1}.
\end{proof}

\end{document}